\documentclass{article}

\usepackage{algpseudocode}
\usepackage{algorithm}
\usepackage{amsmath}
\usepackage{amsfonts}
\usepackage{graphics}
\usepackage{bbm,authblk}
\usepackage{epstopdf}
\usepackage[caption=false]{subfig}

\usepackage[numbers,sort&compress]{natbib}
\bibpunct[, ]{[}{]}{,}{n}{,}{,}

\newtheorem{theorem}{Theorem}[section]
\newtheorem{lemma}[theorem]{Lemma}
\newtheorem{corollary}[theorem]{Corollary}
\newtheorem{proposition}[theorem]{Proposition}

\newtheorem{definition}[theorem]{Definition}
\newtheorem{example}[theorem]{Example}

\newtheorem{remark}{Remark}

\newtheorem{proof}{Proof}

\begin{document}

\title{Dickman type stochastic processes with short- and long- range dependence}

\author{
{Danijel Grahovac\textsuperscript{a}, Anastasiia Kovtun\textsuperscript{b}, Nikolai N. Leonenko\textsuperscript{b}\thanks{CONTACT N.~N. Leonenko. Email: leonenkon@cardiff.ac.uk} and Andrey Pepelyshev\textsuperscript{b}}
}
\affil{\textsuperscript{a}School of Applied Mathematics and Informatics, University of Osijek, Osijek, Croatia;\\ \textsuperscript{b}School of Mathematics, Cardiff University, Cardiff, UK}

\date{}
\maketitle

\begin{abstract}
We study properties of the (generalized) Dickman distribution with two parameters and the stationary solution of the Ornstein-Uhlenbeck stochastic differential equation driven by a Poisson process.
In particular, we show that the marginal distribution of this solution is the Dickman distribution.
Additionally, we investigate superpositions of Ornstein-Uhlenbeck processes which may have short- or long-range dependencies and
marginal distribution of the form of the Dickman distribution. The numerical algorithm for simulation of these processes is presented.
\end{abstract}



\section{Introduction}\label{Introduction}
In \cite{PW} a (generalised) Dickman distribution is defined as a distribution of a random variable $D_\theta$ satisfying the distributional fixed-point equation 
\begin{equation}\label{e:GD-PW}
D_\theta \overset{d}{=} U^{1/\theta} ( 1 +D_\theta),
\end{equation}
where $\theta>0$, $U$ is independent of $D_\theta$ and has the uniform distribution on $(0,1]$. The density of $D_\theta$ is given by
\begin{equation*}
    f_{\theta}(x)=\frac {e^{-\gamma\theta}} {\Gamma(\theta)}\rho_{\theta}(x)\mathbbm{1}_{(0,\infty)}(x), ~x\in\mathbb{R},
\end{equation*}
where $\Gamma(\cdot)$ is the gamma function, $\gamma=-\Gamma'(1)\approx0.5772 $ is Euler's constant  and the function $\rho_{\theta}(x)$ satisfies the difference-differential equation:
\begin{equation}\label{e:rho-de}
\begin{aligned}
    &\rho_{\theta}(x) =0, \quad  x\leq 0,\\
    &\rho_{\theta}(x) =x^{\theta-1}, \quad 0<x\leq 1,\\
    &x \rho'_{\theta}(x) +(1-\theta)\rho_{\theta}(x)+\theta\rho_{\theta}(x-1)=0, \quad  x>1.
\end{aligned}
\end{equation}
It is also known that the random variable $D_\theta$ with the Dickman distribution satisfies
\begin{equation*}
    D_\theta \overset{d}{=} U_1^{1/\theta}+(U_1U_2)^{1/\theta}+(U_1U_2U_3)^{1/\theta}+(U_1U_2U_3U_4)^{1/\theta}+\cdots,
\end{equation*}
where $\{U_k\}_{k=1}^{\infty}$ are independent identically distributed random variables with the uniform distribution on $(0,1]$, see e.g. \cite{PW, pinsky18,GMP}. 

The function $\rho_\theta(x)$ occurs, among others, in context of number theory and combinatorics, see \cite{wheeler90,Franze17,Franze19} and references therein. For $\theta=1$ the function $\rho_1(x)$ is the celebrated Dickman function \cite{D}, which appears even earlier in Ramanujan's unpublished paper \cite{R88}. The Dickman distribution is closely related to the so-called Goncharov distribution and the Poisson-Dirichlet family of distributions, see \cite{pitman1997two,Ivchenko2003GoncharovsMA,PW,MP,ipsen2021generalised} and the references therein. Recently, a number of new applications of the Dickman distribution have appeared in random graph theory \cite{PW}, biology \cite{GR}, physics \cite{Derrida} and other fields; see \cite{MP} for more references and details. The Dickman distribution also appears in various limiting schemes \cite{CO,PW,pinsky18,caravenna2019dickman,Bhattacharjee2019,GMP}.  It also appears as a special case of Vervaat perpetuities \cite{vervaat1979stochastic}.
For the historical account of the Dickman distribution see \cite{bhattacharjee2019dickman,MP} and \cite{weissman2023some}. For simulation from the Dickman distribution see \cite{devroye2010simulating,fill2010perfect,dassios2019exact}.

The term \textit{generalised} in the naming of the distribution defined by \eqref{e:GD-PW} accounts to the presence of the parameter $\theta>0$ and was first used by \cite{PW}. Different generalizations of the Dickman distribution have appeared later on in, for example, \cite{ipsen2021generalised}, see also \cite{handa2009two,bhattacharjee2019dickman,caravenna2019dickman,grabchak2024representation}.

In this paper we first propose a new generalisation of the Dickman distribution by introducing another parameter. The new parameter is the scale parameter of the distribution. We study the properties of the Dickman distribution with the two parameters. 

We then consider the L\'evy driven Ornstein-Uhlenbeck (OU) type processes with Dickman type marginals. Non-Gaussian OU processes and their superpositions were introduced and studied in \cite{BN,BN01,BNS01,BNS}, see also references therein. In the framework of ambit stochastics such stochastic processes have applications in turbulence, financial econometrics, astrophysics, etc.~ \cite{BN18}.

We show that the Poisson driving process yields an OU type process with the Dickman stationary distribution. Furthermore, we consider various extensions with different driving processes that give stationary processes with marginals closely related to the Dickman distribution. In Section \ref{sec4}, we consider superpositions of Dickman OU processes and show their properties. Finally, in Section \ref{sec5} we present simulation methods for the introduced processes.

\section{Generalised Dickman distribution}\label{sec. Dickman distribution}

We start by introducing a generalised Dickman distribution with two parameters.

\begin{definition}\label{D2_def}
A random variable $D_{\theta,a}$ has the \textit{(generalised) Dickman distribution} with parameters $\theta>0$ and $a>0$, shortly $D_{\theta,a} \sim GD(\theta, a)$, if $D_{\theta,a}$ satisfies the distributional fixed-point equation
\begin{equation}\label{e:GDdef}
D_{\theta,a} \overset{d}{=} U^{1/\theta} ( a + D_{\theta,a}),
\end{equation}
where $"\overset{d}{=}"$ denotes the equality in distribution, $U$ is independent of $D_{\theta,a}$ and has the uniform distribution on $(0,1]$.
\end{definition}

For $a=1$, the $GD(\theta,1)$ distribution is equivalent to the generalised Dickman distribution defined by \eqref{e:GD-PW}. Since $a D_{\theta,1} \overset{d}{=} U^{1/\theta} ( a + aD_{\theta,1})$, $aD_{\theta,1}$ satisfies \eqref{e:GDdef}. Hence,
\begin{equation}\label{e:scale}
a D_{\theta,1} \overset{d}{=} D_{\theta,a},
\end{equation}
the parameter $a$ is the scale parameter of the $GD(\theta,a)$ distribution. The next proposition shows that there are many different characterizations of the $GD(\theta,a)$ distribution.

\begin{proposition}
Let $\theta>0$ and $a>0$. 
\begin{enumerate}
\item Let a random variable $D$ satisfy the distributional fixed-point equation
\begin{equation}\label{pr1:i}
D \overset{d}{=} U_a^{1/\theta} ( 1 + a^{-1} D),
\end{equation}
where $U_a$ is independent of $D$ and has the uniform distribution on $(0,a^{\theta}]$. Then $D \sim GD(\theta, a)$.

\item A random variable given by
\begin{equation}\label{e:perpetuity}
D = a U_1^{1/\theta}+ a(U_1U_2)^{1/\theta} + a(U_1 U_2 U_3)^{1/\theta}+\cdots,
\end{equation}
where $U_n, \, n \in \mathbb{N}$, are mutually independent with the uniform distribution on $(0,1]$, has the $GD(\theta, a)$ distribution.

\item A random variable given by
\begin{equation*}
D = \sum_{n=1}^{\infty} a e^{-T_n}
\end{equation*}
where $T_n, \, n \in \mathbb{N}$, are arrival times a Poisson process with parameter $\theta$, has the $GD(\theta, a)$ distribution.
\end{enumerate}
\end{proposition}

\begin{proof}
\begin{enumerate}
\item 
Let $U$  be uniformly distributed on $(0,1]$ and independent of $D$, then from \eqref{pr1:i} we have $D \overset{d}{=} (a^\theta U)^{1/\theta} (1 +  a^{-1} D) \overset{d}{=} U^{1/\theta}(a+D)$, hence, $D$ satisfies \eqref{e:GDdef}.
\item
Note first that the almost sure convergence of the series \eqref{e:perpetuity} was shown in \cite{PW}. From \eqref{e:perpetuity} we get that
\begin{equation*}
D = U_1^{1/\theta} (a + a U_2^{1/\theta} + a (U_2 U_3)^{1/\theta} + \cdots) \overset{d}{=} U_1^{1/\theta} (a + D),
\end{equation*}
hence, $D$ satisfies \eqref{e:GDdef}. 
\item 

This statement follows from Proposition 2 in \cite{PW} and \eqref{e:scale}.
\end{enumerate}
\end{proof}

\begin{remark}
Random variables given by \eqref{e:perpetuity} are referred to as perpetuities, see e.g.~\cite{goldie1996perpetuities} and the references therein. For $a=1$, formula \eqref{e:perpetuity} is known as Vervaat perpetuity \cite{vervaat1979stochastic}. In insurance mathematics, formula \eqref{e:perpetuity} can be interpreted as a present value of payment of amount $a$ every year in the future, subject to random discounting.
\end{remark}

By using \eqref{e:scale}, general properties of the $GD(\theta,a)$ distribution follow readily from Proposition 3 in \cite{PW}.

\begin{proposition}
\begin{enumerate}
\item The Laplace transform of $D_{\theta,a}\sim GD(\theta,a)$ is given by
\begin{equation}\label{e:GDLT}
\psi(s) = \mathbb{E} e^{-sD_{\theta,a}} = \exp \left\{ - \theta \int_0^a (1-e^{-su}) \frac{du}{u}  \right\},\quad s> 0.
\end{equation}

\item Let $\theta_1, \theta_2>0$,  $a>0$,  $D_{\theta_1,a}\sim GD(\theta_1,a)$ and $D_{\theta_2,a}\sim GD(\theta_2,a)$ be independent. Then $D_{\theta_1,a} + D_{\theta_2,a} \sim GD(\theta_1+\theta_2,a)$.

\item The $k$-th cumulant of $D_{\theta,a}\sim GD(\theta,a)$ equals to $a^k \theta /k$. In particular, 
\begin{equation*}
   \mathbb{E}D_{\theta,a}=a\theta,\quad \mathrm{Var}D_{\theta,a}=a^2\frac {\theta} 2.
\end{equation*}
\end{enumerate}
\end{proposition}

We can write \eqref{e:GDLT} as 
$\psi(s) = \exp \left\{ - \theta \cdot \mathrm{Ein}(as) \right\}$,
where
\begin{equation*}
    \mathrm{Ein}(z) = \int\limits_0^z\frac {1-e^{-u}} u du= \sum\limits_{k=1}^{\infty}(-1)^{k-1}\frac {z^k} {k\cdot k!},\quad z\in\mathbb{R}
\end{equation*}
is  the modified exponential integral
introduced by \cite{Schelkunoff1944ProposedSF}, see also \cite{Mainardi2018OnMO} for historical references. 
It follows from \eqref{e:GDLT} that the $GD(\theta,a)$ distribution is infinitely divisible with L\'evy measure given by
\begin{equation*}
    \nu(du)=\frac {\theta} u \mathbbm{1}_{(0,a]}(u)du.
\end{equation*}
\begin{remark}
    The $GD(\theta,a)$ distribution is also self-decomposable with a canonical function $k(x)$ given by $k(x)=\theta\mathbbm{1}_{(0,a]}(x)$ \cite[Corollary 15.11]{S}. However, the $GD(\theta,a)$ distribution does not belong to the Thorin class \cite{James_Thorin_meas} because $k(x)$ is not differentiable.
\end{remark}
Since $a$ is the scale parameter, we have that the density $f_{\theta,a}(x)$ of the $GD(\theta,a)$ distribution satisfies $f_{\theta,a}(x)=a^{-1} f_{\theta,1}(x/a)$ and hence
\begin{equation*}
    f_{\theta,a}(x)=a^{-1}\frac {e^{-\gamma\theta}} {\Gamma(\theta)}\rho_{\theta}(x/a)\mathbbm{1}_{(0,\infty)}(x),
\end{equation*}
where $\rho_{\theta}(x)$ is given by \eqref{e:rho-de}. By using the recurrent  relation for the density $f_{\theta,1}(x)$ obtained in \cite{caravenna2019dickman}, we have that
\begin{equation*}
    f_{\theta,a}(x)=\begin{cases}
        \frac {e^{-\gamma\theta}} {a\Gamma(\theta)} \left(\frac{x}{a}\right)^{\theta-1},& 0<x\leq a,\\
        \frac {e^{-\gamma\theta}} {a\Gamma(\theta)} \left(\frac{x}{a}\right)^{\theta-1}- a^{-1} \theta \left(\frac{x}{a}\right)^{\theta-1}\int_0^{x-a}\frac {f_{\theta,a}(z)} {(1+z/a)^{\theta}}dz,& x>a.
    \end{cases}
\end{equation*}
Another representation of the density $f_{\theta,a}(x)$ is given by
\begin{equation*}
\begin{aligned}
    f_{\theta,a}(x)=\frac {e^{-\gamma\theta}} {a\Gamma(\theta)}\left( \left(\frac{x}{a}\right)^{\theta-1}+\sum_{k=1}^{[x/a]-1}(-\theta)^k K_k(x/a,\theta)\right)\mathbbm{1}_{(0,\infty)}(x),\\
    K_k(x,\theta)=\frac 1 {k!} \idotsint_{D_k(x)}(x-(u_1+\ldots + u_k))^{\theta-1}\frac {du_1\ldots du_k} {u_1 \cdot\ldots\cdot u_k},
\end{aligned}
\end{equation*}
where $[\cdot]$ is the integer part function and $D_k(x)=\{(u_1,\dots,u_k) \in \mathbb{R}^k : u_1+\ldots+u_k\leq x,\; u_1,\ldots u_k\geq 1\}$. This formula was obtained in \cite{vervaat1972} and can also be deduced from Proposition 4.2 of \cite{covo2009one} or from Lemma 1 of \cite{Franze17}. Similar formulas were derived in the context of the Poisson-Dirichlet distribution of population genetics  \cite[Theorem 1]{GR}; see also \cite{wheeler90}.
Figure \ref{density_plots} shows the density $f_{\theta,1}(x)$ for different values of the parameter~$\theta$.

\begin{figure}
\centering
\resizebox*{10cm}{!}{\includegraphics{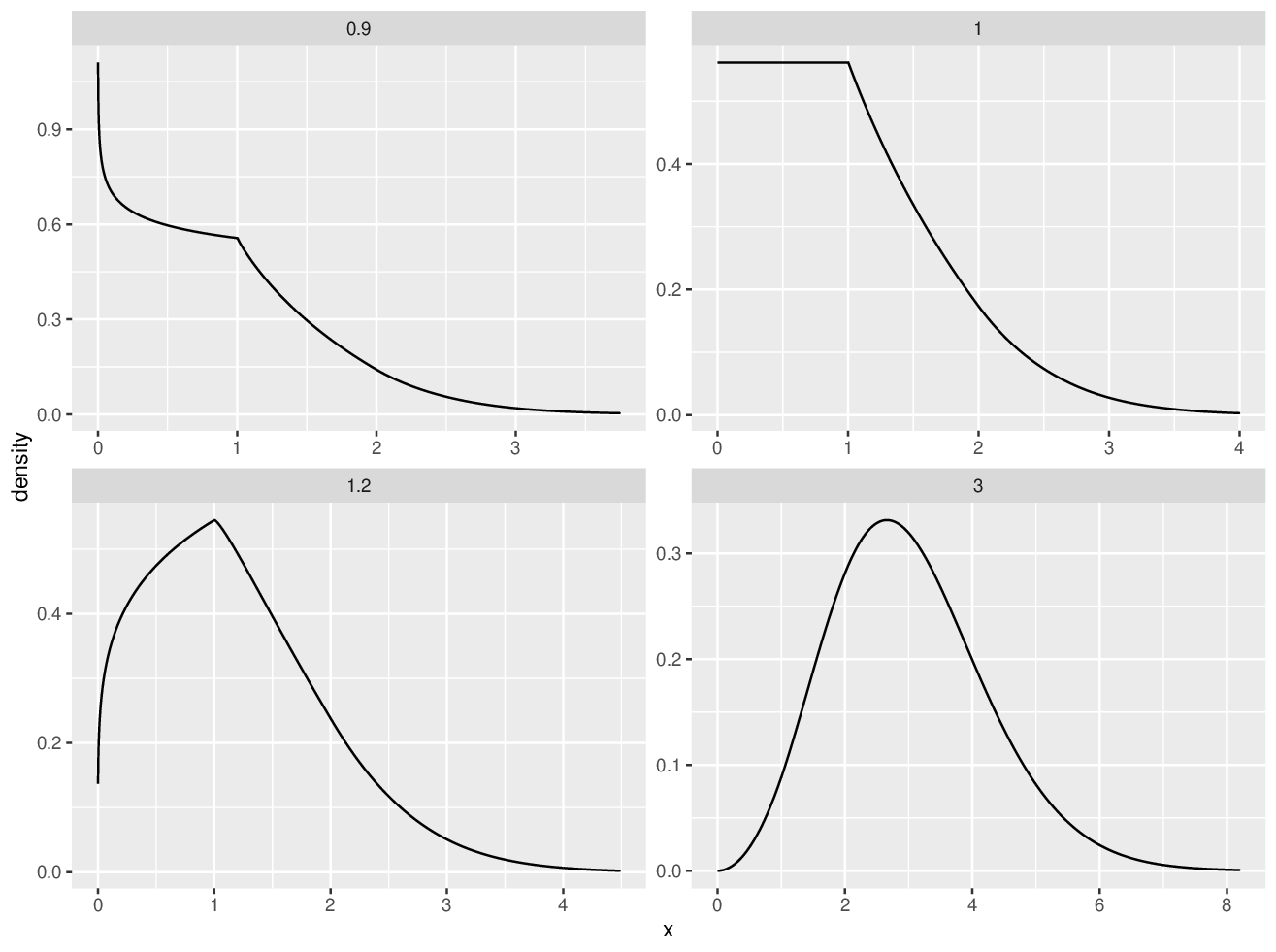}}
\caption{The density $f_{\theta,1}(x)$ of the Dickman distribution for $\theta=0.9,1,1.2 \text{ and } 3$.} \label{density_plots}
\end{figure}

\section{Ornstein-Uhlenbeck type processes with Dickman marginals and beyond}
\label{sec. DOU}
In this section we introduce several stationary processes related to the Dickman distribution. For a stochastic process $Y(t),\;t\geq 0$, we will denote the cumulant function of a random vector $(Y(t_1),\ldots,Y(t_m))$ by
\begin{equation*}
\kappa_{(Y(t_1),\ldots,Y(t_m))}(z_1,\ldots,z_m)= \log \mathbb{E} e^{i (z_1Y(t_1)+\ldots+z_mY(t_m))},\quad z_j\in \mathbb{R},~j=1,\ldots,m,
\end{equation*}
and, in particular,
$
    \kappa_Y(z)=\log\mathbb{E}e^{izY(1)},\; z\in \mathbb{R},
$
is the cumulant function of $Y(1)$.

\subsection{Ornstein-Uhlenbeck process driven by a Poisson process}

Non-Gaussian Ornstein-Uhlenbeck type processes have been studied in \cite{BNS01,BNS}, see also references therein.

A strictly stationary stochastic process $X(t),\;t\geq 0,$ is said to be an \textit{Ornstein-Uhlenbeck process driven by the L\'evy process} (OU type process) if it is the strong solution of the following stochastic differential equation
\begin{equation}\label{OU eq}
    dX(t)=-\lambda X(t)dt+dZ(\lambda t),\quad t\geq 0,
\end{equation}
where $\lambda>0$ and $Z(t)$ is a L\'evy process such that
\begin{equation*}
    \mathbb{E}\log(1+|Z(1)|)<\infty,
\end{equation*}
see \cite{barndorff1998some} for details. The process $Z(t),\;t\geq 0,$ is commonly referred to as the \textit{backward driving L\'evy process} (BDLP), see \cite{BNS01}.

The strictly stationary solution of \eqref{OU eq} is given by 
\begin{equation*}
    X(t)=e^{-\lambda t}X(0)+\int_0^te^{-\lambda (t-\tau)}dZ(\lambda \tau).
\end{equation*}
If we extend $Z(t), \; t \geq 0,$ to a two-sided L\'evy process $Z(t), \; t \in \mathbb{R}$, meaning that we put for $t<0$, $Z(t)=-\widetilde{Z}(-t-)$, where $\widetilde{Z}(t), \; t \geq 0,$ is an independent copy of the L\'evy process $Z(t), \; t \geq 0,$ modified to be c\`adl\`ag, then an OU type process can be written in the form
\begin{equation}\label{OU solutiononR}
    X(t)=\int_{-\infty}^te^{-\lambda (t-\tau)}dZ(\lambda \tau), \quad t\geq 0.
\end{equation}
From \cite{BNS01} and \cite{RAS}, it follows that for any self-decomposable distribution $D$ there exists a BDLP $Z(t),\;t\geq 0$, such that the stationary distribution of an OU type process is $D$.
Moreover, for the cumulant function of $Z$ and $X$ it holds that 
\begin{equation*}
  \kappa_{(X(t_1), \dots, X(t_m))} (z_1,\dots, z_m) = \int_{\mathbb{R}} \lambda \kappa_{Z} \left( \sum_{j=1}^m z_j e^{-\lambda (t_j - s)} \mathbbm{1}_{[0,\infty)} (t_j-s)  \right) ds,
\end{equation*}
see \cite{jurek2001remarks}.
 In particular, using the change of variables we get
\begin{equation}\label{cummulants eq}
   \kappa_{X}(z)=\int_0^{\infty}\kappa_{Z}(e^{-s} z)ds, \quad \kappa_{Z}(z)=z\frac d {dz}\kappa_{X}(z).
\end{equation}

The correlation function of an OU type process (if it exists) is of the form 
\begin{equation*}
    r(\tau)=\mathrm{Corr}(X(t),X(t+\tau))=e^{-\lambda \tau}, \quad \tau\geq0,
\end{equation*}
see e.g.~\cite{S,BNS01,RAS} for more details on OU type processes.

For an OU type processes where the stationary distribution is $GD(\theta,a)$ the cumulant function of the BDLP $Z(t)$ is of the form
\begin{equation}\label{e:PoissontoDickamn-cum}
    \kappa_Z(z)=z\frac d {dz}\int_0^a(e^{izu}-1)\frac {\theta} u du = z\int_0^a e^{izu}(iu)\frac {\theta} u du=\theta (e^{iaz}-1)=\log \mathbb{E}e^{izN_{\theta,a}(1)},
\end{equation}
where $N_{\theta,a}(t),\;t\geq0,$ is a homogeneous Poisson process with rate parameter $\theta>0$ and jumps of size $a>0$, that is $\mathbb{P}(N_{\theta,a}(t)=k)=\mathbb{P}(aN_{\theta,1}(t)=k),\; k\geq0$. Thus, we arrived at the following statement.

\begin{theorem}\label{th_3.1}
    Let $N_{\theta,a}(t),\;t\geq0,$ be a homogeneous Poisson process with rate parameter $\theta$ and jumps of size $a>0$. The stationary solution $X(t)$ of the stochastic differential equation 
    \begin{equation}\label{DOU sde}
        dX(t)=-\lambda X(t)dt+d N_{\theta,a}(\lambda t),\quad  t\geq 0,
    \end{equation}
    has the marginal distribution $GD(\theta,a)$ and  
    \begin{equation}\label{e:DOU-cov}
        \mathrm{Cov}(X(t),X(t+\tau))=e^{-\lambda \tau}\mathrm{Var}X(t)= a^2 \frac {\theta}2e^{-\lambda \tau},\quad \tau\geq 0.
    \end{equation}
\end{theorem}

We will refer to a process $X(t)$ from Theorem \ref{th_3.1} as \textit{Dickman OU (DOU) process}. From \eqref{e:DOU-cov} we get that the second-order spectral density of the DOU process is
\begin{equation*}
    f(\omega)=\frac {a^2\theta} {2\pi} \frac {\lambda} {\lambda^2+\omega^2},\quad \omega\in\mathbb{R}.
\end{equation*}

\begin{remark}
    It is noteworthy that the Ornstein-Uhlenbeck process driven by the Poisson process was first introduced in \cite[Table 2, line 3]{BNS}, where the marginal distribution of such process $X(t), \; t\geq 0$, was not specified but only  the cumulant function was given in the form
    \begin{equation}\label{D 2nd cum}
        \kappa_X(z)=-\theta(\mathrm{Ei}(z)+\log z+\gamma),
    \end{equation}
    where 
    \begin{equation*}
        \mathrm{Ei}(z)=\int_z^{\infty}\frac {e^{-y}} y dy
    \end{equation*}
    is the exponential integral  and $\gamma$ is Euler's constant. The formula (\ref{D 2nd cum}) can be rewritten as the cumulant of the $GD(\theta, 1)$ distribution given in \eqref{e:GDLT}, see, for example, \cite[Appendix D]{mainardi2022fractional}.
\end{remark}

\begin{lemma}
    There exists a temporally homogeneous transition function $P_t(x, B)$
for the DOU process such that
\begin{equation}\label{transition_F}
    \int_{-\infty}^{\infty}e^{izy}P_t(x,dy)=\exp \left\{ iz e^{-\lambda t}x + \frac {\theta} {\lambda} \int_{ae^{-\lambda t}}^{a}(e^{iuz}-1) \frac {du} u \right\}.
\end{equation}
\end{lemma}

\begin{proof}
    Due to Lemma 17.1 in \cite{S} for $X(t)$ being a solution of the equation \eqref{DOU sde} there exists a temporally homogeneous transition function $P_t(x,B)=\mathbb{P}(X(t)\in B \mid X(0)=x)$ on $\mathbb{R}$ such that
\begin{equation*}
    \int_{-\infty}^{\infty}e^{izy}P_t(x,dy)=\exp\left[ize^{-\lambda t}x+\int_0^t\kappa_Z(e^{-\lambda s}z)ds\right], \quad z\in \mathbb{R}
\end{equation*}
and for each $t$ and $x$, $P_t(x,\cdot)$ is infinitely divisible. Taking $Z(t)=N_{\theta,a}(t)$ and inserting its cumulant function $\kappa_{N_{\theta,a}}(z)=\theta (e^{iaz}-1)$, the formula \eqref{transition_F} is obtained by change of variables.
\end{proof}

\begin{remark}
We adopt the convention on the summation that  $\sum\limits_{k=1}^M z_k=0$ when $M=0$.
\end{remark}
From \eqref{transition_F} we immediately get the following statement.

\begin{corollary}
For the process $X(t), \; t\geq 0$, conditionally on $X(0)=x$, for any fixed $t>0$ it holds that
\begin{equation}\label{DOU_Markov_rep}
    X(t)\overset{d}{=} e^{-\lambda t}x+\sum_{k=1}^{\Tilde{N}}\xi_k,
\end{equation}
where $\Tilde{N}$ has the Poisson distribution with parameter $\theta t$ and $\xi_k$, $k=1,2\ldots$, are independent random variables with the density
\begin{equation}\label{CPP_dens}
    g(u)=\frac 1 {\lambda t u}\mathbbm{1}_{(ae^{-\lambda t}, a)}(u).
\end{equation}    
\end{corollary}

\begin{corollary}
The transition function $P_t(x,y)=\mathbb{P}(X(t)\leq y \mid X(0)=x)$ of the DOU process can be represented in the form
\begin{equation*}
    P_t(x,y)=\begin{cases}
        0, & y<e^{-\lambda t}x\\
        e^{- \theta t}, & y=e^{-\lambda t}x\\
        \sum_{n=1}^{\infty}\frac {( t \theta)^n} {n!}e^{- \theta t}\int_0^{y-e^{-\lambda t }x}g_n(u) du,& y>e^{-\lambda t}x,
    \end{cases}
\end{equation*}
where $g_n(u)=\int_0^ug(y)g_{n-1}(u-y)dy, \;n\geq 2$ and $g_1(u)=g(u)$ is given by (\ref{CPP_dens}).
\end{corollary}

The above corollaries can also be deduced from \cite{zhang2011}. 
The DOU process is exponentially $\beta$-mixing \cite[Corollary 4.4]{M}. Hence, from \cite[Theorem 19.2]{B} it follows that as $T\rightarrow \infty$
\begin{equation*}
    X_T(t)\Rightarrow B(t),\quad t\in(0,1],
\end{equation*}
where $B(t),\;t\geq 0,$ is the standard Brownian motion, $\Rightarrow$ denotes weak convergence in the space $D[0,1]$ of c\`adl\`ag functions with Skorokhod topology, and 
\begin{equation*}
    X_T(t)=\frac 1 {s_c \sqrt{T}}\int_0^{Tt}(X(\tau)-a\theta)d\tau\quad\text{or} \quad
    X_T(t)= \frac 1 {s_d \sqrt{T}}\sum_{\tau=1}^{[Tt]}(X(\tau)-a\theta),\quad  t\in[0,1],
\end{equation*}
where $s_c^2=\frac {a^2\theta} {\lambda}$ in view of \eqref{e:DOU-cov} and $s_d^2= \frac{a^2\theta}{2} \frac{1+e^{-\lambda}}{1-e^{-\lambda}}$.

In the same manner, the stationary autoregressive process of order $1$ with $GD(\theta,a)$ marginals can be constructed. Let $X_0$ be a random variable with the $GD(\theta,a)$ distribution and $0<c<1$. We define an autoregressive process $X_1,X_2,\ldots$ as
\begin{equation*}
    X_{n}=cX_{n-1}+\varepsilon_{n}, \quad  n\geq 1,
\end{equation*}
where the innovation process $\varepsilon_n, \; n\geq 1$, is a sequence of independent identically distributed random variables, independent of $X_0$, and 
\begin{equation*}
    \varepsilon_1=c \int_0^{-\log c} e^{-\tau}dN_{\theta,a}(\tau),
\end{equation*}
where $N_{\theta,a}(t)$ is the homogeneous Poisson process with parameter $\theta>0$ and jumps of size $a>0$.
Moreover, 
\begin{equation*}
   \varepsilon_1\overset{d}{=}\sum_{k=1}^{\Tilde{N}}a\xi_k,
\end{equation*}
where $\Tilde{N}$ has the Poisson distribution with parameter $\theta $ and $\xi_k$, $k=1,2\ldots$, are independent random variables with the density
\begin{equation*}
    g(u)=\frac 1 { u\log 1/c}\mathbbm{1}_{(c, 1)}(u).
\end{equation*}    
Then $\{X_n\}_{n=1}^{\infty}$ is a strictly stationary process with marginals given by the $GD(\theta,a)$ distribution and the covariance function 
$
    \mathrm{Cov}(X_n,X_{n+\tau})= a^2 \frac {\theta} 2 c^{\tau},~ \tau=0,1,\ldots,
$
while the spectral density is of the form 
\begin{equation*}
    f(\omega)=\frac {a^2\theta} {2\pi} \sum_{k\in\mathbb{Z}}\frac {-\log c} {\log^2 c+(\omega+2\pi k)^2},  \quad \omega\in[-\pi,\pi).
\end{equation*}

\subsection{Ornstein-Uhlenbeck process driven by a Poisson process of order $k$}

The Poisson process of order $k$ was introduced in \cite{KO}, see also \cite{KL,KLS} and the references therein.
The \textit{Poisson process of order $k$}, $k\geq 1$, is given by
\begin{equation*}
    N_{\theta}^{(k)}(t)=\sum_{i=1}^{N_{k\theta}(t)}Y_i,\quad t\geq 0,
\end{equation*}
where $Y_i, \; i=1,2,\dots$ is a sequence of independent identically distributed random variables with the discrete uniform distribution on the set $\{1,2,\ldots,k\}$, which is independent of the homogeneous Poisson process $N_{k\theta}(t)$ with parameter $k\theta$.
If we take the Poisson process of order $k$ as the BDLP in \eqref{OU eq}, we get an OU type process which is closely related to the Dickman OU process. 

\begin{theorem}\label{th3.2}
Let $X^{(k)}(t),\;t\geq 0,$ be a stationary solution of the stochastic differential equation (\ref{OU eq}) with the BDLP $Z(t)=N_{\theta}^{(k)}(t),\;t\geq 0$, the Poisson process of order $k$. Then the following equality of finite dimensional distributions holds
\begin{equation}\label{e:DOUPPk}
    \{X^{(k)}(t), \; t \geq 0  \}\overset{\text{fdd}}{=} \left\{ \sum_{j=1}^k X_{\theta, j}(t), \; t \geq 0 \right\},
\end{equation}
where $X_{\theta,j}(t),\; j=1,\ldots,k$, are independent DOU processes with parameters $\theta$ and $a_j=j$. Moreover,
    \begin{align*}
          &\mathbb{E}X^{(k)}(t)=\theta\frac {k(k+1)} 2,\\
        &\mathrm{Cov}(X^{(k)}(t),X^{(k)}(t+\tau))=\frac {\theta}2\frac {k(k+1)(2k+1)} 6 e^{-\lambda \tau},\;\tau\geq 0.
    \end{align*}
\end{theorem}

\begin{proof}
For $0< t_1< \cdots < t_m$, the cumulant function of the random vector $(X^{(k)}(t_1), \dots, X^{(k)}(t_m))$ is
\begin{align*}
\kappa_{(X^{(k)}(t_1), \dots, X^{(k)}(t_m))} (z_1,\dots, z_m) = \log \mathbb{E} e^{i (z_1 X^{(k)}(t_1) + \cdots + z_m X^{(k)}(t_m))}
=&\\ \int_{\mathbb{R}} \lambda \kappa_{N_{\theta}^{(k)}} \left(\sum_{l=1}^m z_l e^{-\lambda(t_l-s)} \mathbbm{1}_{[0,\infty)} (t_l-s) \right) ds &,
\end{align*}
see e.g.~Eq.~(3.7) in \cite{taufer2011characteristic}.
The cumulant function of $N_\theta^{(k)}(1)$ is
\begin{equation*}
    \kappa_{N_{\theta}^{(k)}}(z)= k\theta \left( \sum_{j=1}^{k}\frac{1}{k} e^{ijz} -1 \right) = \theta \sum_{j=1}^{k}(e^{ijz} - 1).
\end{equation*}
Hence, we get that
\begin{align*}
\kappa_{(X^{(k)}(t_1), \dots, X^{(k)}(t_m))} (z_1,\dots, z_m) &= \sum_{j=1}^k \int_{\mathbb{R}} \lambda \theta \left( e^{ij\sum_{l=1}^m z_l e^{-\lambda(t_l-s)} \mathbbm{1}_{[0,\infty)} (t_l-s)} - 1 \right) ds\\
&= \sum_{j=1}^k \int_{\mathbb{R}} \lambda \kappa_{N_{\theta,j}} \left( \sum_{l=1}^m z_l e^{-\lambda(t_l-s)} \mathbbm{1}_{[0,\infty)} (t_l-s) \right) ds\\
&= \sum_{j=1}^k \log \mathbb{E} e^{i (z_1 X_{\theta,j}(t_1) + \cdots + z_m X_{\theta,j}(t_m))},
\end{align*}
where for each $j=1,\dots,k$, $X_{\theta,j}(t)$ is the DOU processes with parameters $\theta$ and $j$,  $N_{\theta,j}(t)$ is its driving Poisson process with rate parameter $\theta$ and jumps of size $j$. This proves \eqref{e:DOUPPk}.
\end{proof}

\subsection{OU processes driven by Bell-Touchard processes}
Following \cite{FR}, a Bell-Touchard process with parameter $\alpha>0$ and $\nu>0$ is a compound Poisson process
\begin{equation}\label{BT}
    B_{\alpha,\nu}(t)=\sum_{i=1}^{N_{\alpha,\nu}(t)} Y_i , \quad t\geq 0,
\end{equation}
 where $Y_i, \; i=1,2,\dots$, is the sequence of i.i.d.~random variables with the probability mass function
\begin{equation*}
    \mathbb{P}(Y_i=n)=\frac 1 {(e^{\nu}-1)} \frac {\nu^n} {n!}, \quad n=1,2,\ldots,
\end{equation*}
which are independent of the Poisson process $N_{\alpha,\nu}(t)$ with parameter $\alpha(e^{\nu}-1)$. 
\begin{remark}
    Note that the distribution of the jumps $Y_i, i=1,2,\ldots$, is the zero-truncated Poisson distribution.
\end{remark}
Similarly as in the previous subsection, we consider the OU process \eqref{OU eq} with the BDLP $Z(t)=B_{\alpha,\nu}(t),\;t\geq 0$, given by \eqref{BT}. It turns out that this process corresponds to an infinite superposition of independent DOU processes.

\begin{theorem}
Let $X^{BT}(t),\;t\geq 0,$ be a stationary solution of the stochastic differential equation (\ref{OU eq}) with the BDLP $Z(t)=B_{\alpha,\nu}(t),\;t\geq 0$, the Bell-Touchard process. Then
\begin{equation*}
     \left\{X^{BT}(t), \; t \geq 0 \right\} \overset{\text{fdd}}{=} \left\{\sum_{j=1}^{\infty} X_{\theta_j,j}(t), \; t \geq 0 \right\},
\end{equation*}
where $X_{\theta_j,j}(t),\; j\geq 1$, are independent DOU processes with parameters $\theta_j=\frac {\alpha\nu^j} {j!}$ and $a_j=j$. Mean and covariance functions of $X_{\alpha,\nu}(t),\; t \geq 0$, are given by
\begin{align*}
    &\mathbb{E}X^{BT}(t)=\sum_{j=1}^{\infty}j\frac{\alpha \nu^j} {j!} =\alpha\nu e^{\nu},\\
    &\mathrm{Cov}(X^{BT}(t),X^{BT}(t+\tau))=\frac {\alpha e^{\nu}} {2} \nu(\nu+1) e^{-\lambda \tau}.
\end{align*}
\end{theorem}

\begin{proof}
Since
\begin{equation*}
\kappa_{B_{\alpha,\nu}}(z) = \alpha(e^\nu -1) \left( \sum_{j=1}^\infty e^{i j z } \frac 1 {(e^{\nu}-1)} \frac {\nu^j} {j!} - 1\right) = \alpha \sum_{j=1}^\infty \frac {\nu^j} {j!} \left(e^{i j z} - 1\right),
\end{equation*}
similarly as in the proof of Theorem \ref{th3.2} we have that
\begin{align*}
\kappa_{(X^{BT}(t_1), \dots, X^{BT}(t_m))} (z_1,\dots, z_m) &= \int_{\mathbb{R}} \lambda \kappa_{B_{\alpha,\nu}} \left(\sum_{l=1}^m z_l e^{-\lambda(t_l-s)} \mathbbm{1}_{[0,\infty)} (t_l-s) \right) ds\\
&=\sum_{j=1}^\infty \int_{\mathbb{R}} \lambda \alpha \frac {\nu^j} {j!} \left(e^{i \sum_{l=1}^m j z_l e^{-\lambda(t_l-s)} \mathbbm{1}_{[0,\infty)} (t_l-s)} - 1\right)  ds\\
&= \sum_{j=1}^\infty \int_{\mathbb{R}} \lambda \kappa_{N_{\alpha  \frac {\nu^j} {j!},j}} \left( \sum_{l=1}^n z_l e^{-\lambda(t_l-s)} \mathbbm{1}_{[0,\infty)} (t_l-s) \right) ds\\
&= \sum_{j=1}^\infty \log \mathbb{E} e^{i (z_1 X_{\theta_j,j}(t_1) + \cdots + z_n j X_{\theta_j,j}(t_n))},
\end{align*}
where for each $j=1,\dots,k$, $X_{\theta_j,j}(t), \; t\geq 0$, is the DOU process with parameters $\theta_j=\frac {\alpha\nu^j} {j!}$ and $a_j=j$,  driven by the Poisson process  $N_{\alpha  \frac {\nu^j} {j!},j}(t)$ with rate parameter $\alpha  \frac {\nu^j} {j!}$ and jumps of size $j$. 
\end{proof}

\section{Superpositions of DOU processes}\label{sec4}
While OU type processes provide stationary models with flexible choice of marginal distributions, the correlation function of these processes always decays exponentially. However, by considering superpositions of such processes, different correlation structures may be obtained. 
Superpositions of OU type processes have been introduced in \cite{BN01}, see also \cite{barndorff2011multivariate,BN18}. The basic idea of the construction is to randomize the parameter $\lambda$ in \eqref{OU solutiononR}. Instead of following \cite{BN01}, we shall introduce supOU processes using the parametrization from \cite{fasen2005extremes} which is more suitable for simulation and corresponds better to the definition \eqref{OU solutiononR}.

Let $Z(t), \; t \geq 0,$ be a L\'evy process such that $\mathbb{E}\log(1+|Z(1)|)<\infty$ and having the L\'evy-Khintchine characteristic triplet $(b,\sigma^2,\mu)$ such that
\begin{equation*}
    \kappa_Z(z)=izb-\frac 1 2z^2\sigma^2+\int_{\mathbb{R}}(e^{izx}-1-izx\mathbbm{1}_{[-1,1]}(x))\mu(dx).
\end{equation*}
The second ingredient in the construction of the superposition process is a measure $\pi$ on $\mathbb{R}^+$ such that 
\begin{equation}\label{e:eta}
\eta^{-1} := \int_{\mathbb{R}_+} \xi^{-1} \pi(d\xi)<\infty.
\end{equation}
The quadruple $(b,\sigma^2,\mu,\pi)$ is sometimes called the \textit{generating quadruple} \cite{fasen2005extremes}.
Suppose that L\'evy basis $\Lambda$ is a homogeneous infinitely divisible independently scattered random measure  on $S=\mathbb{R}_+ \times \mathbb{R}$ such that
\begin{equation*}
    \kappa_{\Lambda(A)}(z)=\log\mathbb{E}e^{iz\Lambda(A)}=(\pi\times Leb)(A)\kappa_Z(z), \quad A\in\mathcal{B}(\mathbb{R}_+\times\mathbb{R}),
\end{equation*}
where $Leb$ denotes the Lebesgue measure on $\mathbb{R}$. A \textit{superposition of OU type processes} (\textit{supOU}) is defined as the following stochastic integral with respect to $\Lambda$ \cite{BN01, fasen2005extremes}
\begin{equation}\label{supOU}
    Y(t)=\int_{\mathbb{R}_+}\int_{\mathbb{R}}e^{-\xi (t-s)}\mathbbm{1}_{[0,\infty)}(t-s)\Lambda(d\xi,\eta ds).
\end{equation}
The existence of the integral in \eqref{supOU} was proven in \cite{BN01} and the equivalence of different parametrizations was shown in \cite{fasen2005extremes}, see also \cite{barndorff2011multivariate}. The supOU process is strictly stationary and its stationary distribution is a self-decomposable distribution determined by the BDLP $Z(t)$ similarly as with the OU type processes. Moreover, for $0\leq t_1< \cdots < t_m$, the cumulant function of finite dimensional distributions is given by
\begin{equation}\label{cumfidissupOU}
\kappa_{(Y(t_1), \dots, Y(t_m))} (z_1,\dots, z_m) = \int_{\mathbb{R}_+} \int_{\mathbb{R}} \eta \kappa_{Z} \left(\sum_{l=1}^m z_l e^{-\xi(t_l-s)} \mathbbm{1}_{[0,\infty)} (t_l-s) \right) ds \pi(d\xi).
\end{equation}
In particular, the relation \eqref{cummulants eq} between cumulant functions holds. The correlation function of $Y(t)$, if it exists, is given by
\begin{equation}\label{e:supOUcorr}
    r(\tau)=\mathrm{Corr}(Y(t),Y(t+\tau))=\eta \int_0^{\infty}\xi^{-1} e^{-\tau\xi}\pi(d\xi),\quad \tau\geq 0.
\end{equation}
Thus, it follows that
\begin{equation}\label{int_corr}
    \int_0^{\infty}r(\tau)d\tau=\eta \int_0^{\infty} \xi^{-2} \pi(d\xi),
\end{equation}
and this integral can be both finite and infinite. Hence, we will say that a supOU process exhibits long-range dependence (has a long memory) if the integral in \eqref{int_corr} is infinite and we will say that it exhibits a short-range dependence (short memory) otherwise. Moreover, if $\pi$ is regularly varying at zero, i.e.
\begin{equation}\label{e:regvarofpi}
\pi((0,\xi]) \sim \ell(\xi^{-1}) \xi^{1+\alpha}, \quad \text{ as } \xi \to 0,
\end{equation}
for some $\alpha>0$ and a slowly varying function $\ell(\cdot)$ at infinity, then
\begin{equation*}
r (\tau) \sim \frac{\eta (1+\alpha)}{\Gamma(\alpha)} \ell(\tau) \tau^{-\alpha}, \quad \text{ as } \tau \to \infty,
\end{equation*}
see \cite{GLT}.
In particular, if $\alpha \in (0,1)$ in \eqref{e:regvarofpi}, then the supOU process exhibits long-range dependence; see \cite{fasen2005extremes} and \cite{GLST} for details.
Overall, we can see that the supOU processes form a wide class of stochastic processes and the representation \eqref{supOU} enables one to independently model marginal distributions by the choice of the BDLP and the correlation structure by the choice of the measure $\pi$.

\subsection{Dickman supOU processes}

Let $\pi$ be any measure on $\mathbb{R}_+$ such that \eqref{e:eta} holds and let $Z(t), \; t \geq 0,$  be a Poisson process with parameter $\theta$ and jumps of size $a>0$ such that
$
\kappa_Z(z) = \theta (e^{iaz} - 1),
$
and the L\'evy-Khintchine triplet is
\begin{equation}\label{Poisson_characteristics}
     b=\begin{cases} a\theta, &\text{if } a\leq1\\
                    0, &\text{if }a>1\\
     \end{cases}, \quad \sigma^2=0, \quad \mu=\theta\delta_a,
\end{equation}
where $\delta_a$ is the Dirac measure concentrated at $a$. 

It follows from \eqref{e:PoissontoDickamn-cum} that the supOU process $Y(t),\; t \geq 0,$ with generating quadruple $(b,0,\mu,\pi)$ has the $GD(\theta,a)$ marginal distribution and the correlation function \eqref{e:supOUcorr}. We will refer to this class of supOU processes as \textit{Dickman supOU processes} (\textit{supDOU}). Note that for any choice of the measure $\pi$ on $\mathbb{R}_+$ we get one example of a stationary process with Dickman marginals. We now consider some examples for specific choices of $\pi$.

\begin{example}\label{ex.1}
Suppose that $\pi$ is degenerate such that $\pi\left(\{\lambda\}\right)=1$ for some $\lambda>0$. Then it follows from \eqref{cumfidissupOU} since $\eta=\lambda$ that the finite dimensional distributions of the supDOU process are the same as for the standard OU type process \eqref{OU solutiononR}, that is
\begin{equation*}
  \kappa_{(Y(t_1), \dots, Y(t_m))} (z_1,\dots, z_m) = \int_{\mathbb{R}} \lambda \kappa_Z \left( \sum_{j=1}^m z_j e^{-\lambda (t_j - s)} \mathbbm{1}_{[0,\infty)} (t_j-s)  \right) ds.
\end{equation*}
\end{example}

\begin{example}\label{ex.2}
Suppose $\pi$ is a discrete probability measure such that $\pi\left(\{\lambda_j\}\right)=p_j$, $j \in \mathbb{N}$, $\lambda_j>0$ and suppose that \eqref{e:eta} holds, i.e.
\begin{equation*}
\eta^{-1} = \sum_{j=1}^{\infty} \lambda_j^{-1} p_j < \infty. 
\end{equation*}
Then we have from \eqref{cumfidissupOU} that the supDOU process has the cumulant function of finite dimensional distributions given by
\begin{align*}
\kappa_{(Y(t_1), \dots, Y(t_m))} (z_1,\dots, z_m) &= \sum_{j=1}^\infty \int_{\mathbb{R}} \eta  p_j \kappa_{Z} \left(\sum_{l=1}^m z_l e^{-\lambda_j (t_l-s)} \mathbbm{1}_{[0,\infty)} (t_l-s) \right) ds\\
&= \sum_{j=1}^\infty \int_{\mathbb{R}} \lambda_j \frac{\eta  p_j}{\lambda_j} \kappa_{Z} \left(\sum_{l=1}^m z_l e^{-\lambda_j (t_l-s)} \mathbbm{1}_{[0,\infty)} (t_l-s) \right) ds.
\end{align*}
Hence, a supDOU process $Y(t), \; t \geq 0,$ has the same finite dimensional distributions as
\begin{equation*}
\left\{Y(t), \; t \geq 0 \right\} \overset{\text{fdd}}{=} \left\{\sum_{j=1}^{\infty} X^{(\lambda_j)}_{\theta_j, a}(t),  \; t \geq 0\right\},
\end{equation*}
where $X^{(\lambda_j)}_{\theta_j, a}(t), \; t \geq 0$, are independent DOU processes with parameters $\theta_j=\frac{\eta p_j}{\lambda_j}\theta$, $a_j=a$ and mean-reverting parameter $\lambda_j$.
\end{example}

\begin{example}\label{ex.3}
Let $\pi$ be the gamma distribution $Gamma(1+\alpha, \beta)$ for some $\alpha>0$ and $\beta>0$ given by the density
\begin{equation*}
p(\xi) = \frac{\beta^{1+\alpha}}{\Gamma(1+\alpha)} \xi^{\alpha} e^{-\beta\xi},\quad \xi>0,
\end{equation*}
where $\Gamma(\cdot)$ is the gamma function. Then \eqref{e:eta} holds and a resulting supDOU process has $GD(\theta,a)$ marginals and the correlation function can be explicitly computed to be
\begin{equation*}
r(t) = \left( 1 + \frac{t}{\beta}\right)^{-\alpha},\quad t\geq 0.
\end{equation*}
In particular, for $\alpha \in (0,1)$ we obtain a long-range dependent process with Dickman marginals.
\end{example}
\noindent
More examples of possible choices of $\pi$ and corresponding supOU processes are given in \cite{BNL}.

\subsection{Limit theorems for supDOU processes}
Let $Y(t), \; t \geq 0,$ be a supDOU process with parameters $\theta>0$ and $a>0$. Denote $Y^*(t), \; t \geq 0,$ the integrated supDOU process
\begin{equation*}
    Y^*(t)=\int_0^t (Y(u)-\theta)du,
\end{equation*}
Limit theorems for integrated supOU processes have been proved in \cite{GLT} for the finite variance case and in \cite{grahovac2020multifaceted} for the infinite variance case. We apply now these results to supDOU processes.
In the following, $\overset{\text{fdd}}{\to}$ denotes convergence of all finite-dimensional distributions and de Bruijn conjugate $\ell^{\sharp}(\cdot)$ of some slowly varying function $\ell(\cdot)$ which is a unique slowly varying function such that 
$\ell(x)\ell^{\sharp}(x \ell(x)) \rightarrow 1 $ and 
$\ell^{\sharp}(x) \ell(x \ell^{\sharp}(x))\rightarrow 1$ as $x\rightarrow \infty$
see \cite[Theorem 1.5.13]{bingham1989regular}.

\begin{corollary}\label{cor:limitthm}

\begin{enumerate}
\item If $\int_0^{\infty} \xi^{-2} \pi(d\xi)<\infty$, then
\begin{equation*}
\left\{ \frac{1}{T^{1/2}}Y^*(Tt)\right\} \overset{\text{fdd}}{\to} \tilde{\sigma}B(t) \quad \text{ as } T \to \infty, 
\end{equation*}
where $\{B(t), t\geq 0\}$ is the standard Brownian motion and  \\ $\tilde{\sigma}^2 = \eta a^2 \theta \int_{0}^{\infty} \xi^{-2} \pi(d\xi)$.
\item If $\int_0^{\infty} \xi^{-2} \pi(d\xi)=\infty$, assume that $\pi$ has a density $p(x)$ such that for some $\alpha\in(0,1)$ and $\ell(\cdot)$ is slowly varying at infinity
\begin{equation*}
p(x)\sim (1+\alpha) \ell(x^{-1})x^{\alpha}, \quad as \; x\rightarrow 0.
\end{equation*}
Then
\begin{equation*}
\Bigg\{\frac 1
{T^{1/(1+\alpha)}\ell^{\sharp}(T)^{1/(1+\alpha)}}Y^*(Tt)%
\Bigg\}\xrightarrow{\text{fdd}} \{S_{1+\alpha}(t)\}, \quad \text{ as } T \to \infty,
\end{equation*}
where $\ell^{\sharp}(\cdot)$ is de Bruijn conjugate of $1/\ell(x^{1/(1+\alpha)})$ and $%
\{S_{1+\alpha}(t)\}$ is the $(1+\alpha)$-stable L\'evy process such that 
\begin{align*}
\kappa_{S_{1+\alpha}(1)}(z)=-|z|^{1+\alpha} \frac{\Gamma(1-\alpha)}{-\alpha} \theta \eta^{-1} a^{1+\alpha} \cos\left(\frac{\pi(1+\alpha)}{2}\right)  \times \\ \times \left(1 - i\, \mathrm{sign}(z) \tan \left(\frac{\pi (1+\alpha)}{2}\right) \right).
\end{align*}
\end{enumerate}
\end{corollary}

\begin{proof}
The statement follows from Theorems 3.2 and 3.4 in \cite{GLT} since the supDOU process has finite variance and the Blumenthal-Getoor index of the BDLP is $0$.
\end{proof}

The convergence in Corollary \ref{cor:limitthm}(i) can be extended to weak convergence in the space of continuous functions $C[0,1]$ with uniform topology provided that additionally the fourth moment is finite (which is always true for the Dickman distribution) and $\int_0^{\infty} \xi^{-3} \pi(d\xi)<\infty$, see \cite{GLT} for details. 

Integrated supOU processes with long-range dependence exhibit another interesting limiting behaviour called intermittency, see \cite{GLT}. 
\begin{definition}
 For a stochastic process $X(t), t\geq 0$, let
 \begin{equation*}
    \tau_{X}(q)=\lim_{t\to \infty} \frac {\log \mathbb{E}|X(t)|^q} {\log t}
\end{equation*}
denote its scaling function which measures the rate of growth of moments as time goes to infinity.
We will say that the process $X(t)$ exhibits the intermittency property if there exists $q^*$ such that the function $q\rightarrow \tau_X(q)/q$ is strictly increasing on $[q^*, \infty)$.
\end{definition} 
For the integrated supDOU processes $Y^*(t), t\geq 0$, the scaling function $\tau_{Y^*}(q)/q$ is constant if, for example, the measure $\pi$ is concentrated on a finite set; specifically, we have
\begin{equation*}
\tau_{Y^*}(q) = \frac{1}{2} q.
\end{equation*}
However, under the assumptions of Corollary \ref{cor:limitthm}(i), if we additionally assume that \eqref{e:regvarofpi} holds with $\alpha>1$, then the scaling function has the form
\begin{equation*}
    \tau_{Y^*}(q)=\begin{cases}
        \frac 1 {2}q, & 0<q\leq 2\alpha, \\
        q-\alpha, & q\geq 2\alpha,
    \end{cases}
\end{equation*}
and under the assumptions of Corollary \ref{cor:limitthm}(ii) the scaling function has the form
\begin{equation*}
    \tau_{Y^*}(q)=\begin{cases}
        \frac 1 {1+\alpha}q, & 0<q\leq 1+\alpha, \\
        q-\alpha, & q\geq 1+\alpha.
    \end{cases}
\end{equation*}
Since in both cases $\tau_{Y^*}(q)/q$ is strictly increasing for $q>2$ and $q>1+\alpha$ respectively, such integrated supDOU processes exhibit intermittency, see \cite{GLST,GLST2016JSP,GLT,grahovac2021intermittency} for details.

\section{Simulations}\label{sec5}
For the simulation of the DOU process \eqref{DOU sde} we use Algorithm 1 which is based on the Markovian property of the OU process and representation \eqref{DOU_Markov_rep}. Figure \ref{DOU_sim} shows the simulated paths with parameters $a=1$, $\theta=3$, $\lambda=1$ on  intervals $[0,10]$ and $[0,300]$.

\begin{algorithm}
\caption{Simulation of the DOU process}
\begin{algorithmic}[1]
\Function{DOU\_Random}{$\theta$, $\lambda$, $T$, $dt$}
    \State $n \gets 1 + \lceil T / dt \rceil$
    \State $OU \gets \text{rep}(0,n)$
    \State $OU[1] \gets \text{DickmanRandom}(\theta)$
    \State $time \gets 0$
    \State $\tau \gets \text{ExponentialRandom}(\theta * \lambda)$
    
    \For{$i \gets 1$ {\bf to} $n - 1$}
        \State $time \gets time + dt$
        \If{$time \geq \tau$}
            \State $OU[i + 1] \gets \exp(-\lambda * dt) * OU[i] + 1$
            \State $\tau \gets \tau + \text{ExponentialRandom}(\theta * \lambda)$
        \Else
            \State $OU[i + 1] \gets \exp(-\lambda * dt) * OU[i]$
        \EndIf
    \EndFor
    
    \State \Return $OU$
\EndFunction
\end{algorithmic}
\end{algorithm}

\begin{figure}
\centering
\resizebox*{12cm}{!}{\includegraphics{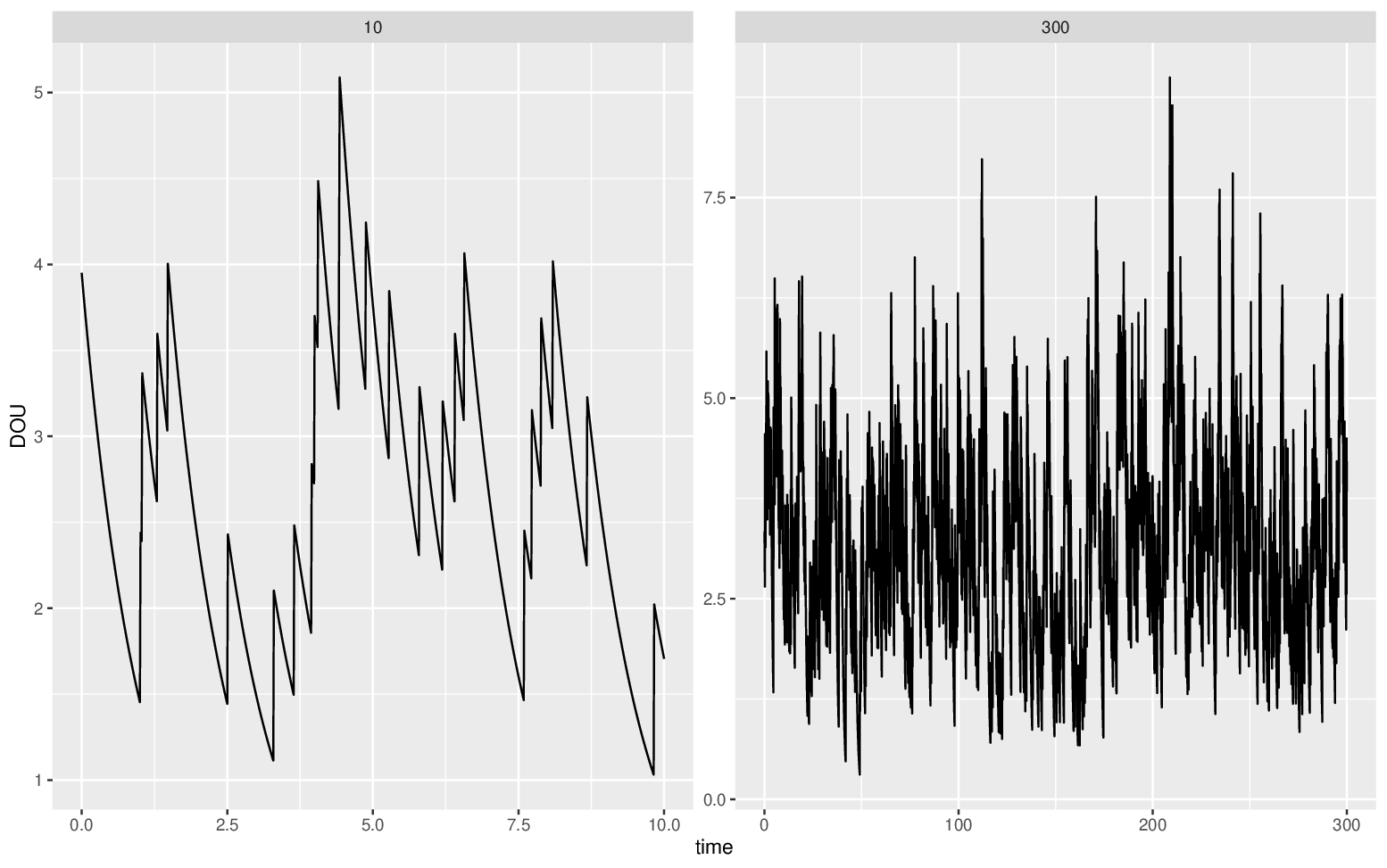}}
\caption{Sample trajectories of the DOU process with $a=1$, $\theta=3,$ $\lambda=1$ on the interval $[0,10]$ (left) and the interval $[0,300]$ (right).} \label{DOU_sim}
\end{figure}

Next, we consider the supDOU processes. The L\'evy basis $\Lambda$ in the definition of the supDOU process has a generating quadruple $(b,0,\mu,\pi)$ given by \eqref{Poisson_characteristics} with arbitrary measure $\pi$ satisfying \eqref{e:eta}. Using \cite[Theorem 2.2]{barndorff2011multivariate}, we consider a modification of $\Lambda$ such that for $A \in \mathcal{B}(\mathbb{R}_+\times \mathbb{R})$ we have
\begin{equation*}
\Lambda(A) = 
\left(b -  \int_{|x|\leq 1} x \theta\delta_a(dx) \right) (\pi \times Leb)(A) + \int_A \int_{\mathbb{R}}  x N(d\xi, ds, dx),
\end{equation*}
where $N$ is an independent Poisson random measure on $\mathbb{R}_- \times \mathbb{R} \times \mathbb{R}$ with intensity $\pi \times Leb \times \mu$. Hence we get that
\begin{equation*}
\Lambda(A) = \sum_{k=-\infty}^\infty a \delta_{(R_k, S_k)} (A),
\end{equation*}
where $-\infty<\cdots<S_{-1} < S_0 \leq 0 < S_1 < \cdots<\infty$ are the jump times of a two-sided Poisson process on $\mathbb{R}$ with intensity $\theta$ and $\{R_k, \, k \in \mathbb{Z}\}$ is an i.i.d.~sequence with the distribution $\pi$ independent of $\{S_k,\;k \in \mathbb{Z}\}$, see also \cite{fasen2005extremes}.

Therefore, the supDOU process can be represented as
\begin{equation}\label{supOU:sum_rep}
    Y(t)=\sum_{k=-\infty}^{\infty}ae^{-R_k (t-S_k)}\mathbbm{1}_{[0,\infty)}(t-S_k),
\end{equation}
which is convenient for simulation of the supDOU process. 

In case  of a supDOU process with a discrete measure $\pi$, see Example \ref{ex.2}, we have
\begin{equation*}
\left\{Y(t), \; t \geq 0 \right\} \overset{\text{fdd}}{=} \left\{\sum_{j=1}^{\infty} X^{(\lambda_j)}_{\theta_j, a}(t), \; t \geq 0\right\},
\end{equation*}
where $X^{(\lambda_j)}_{\theta_j, a}(t), \; t \geq 0,$ are independent supDOU processes with parameters $\theta_j=\frac{\eta p_j}{\lambda_j}\theta$, $a_j=a$ and mean-reverting parameter $\lambda_j$. The trajectories of this process can be approximated using finite sums or \eqref{supOU:sum_rep}, where the sequence $\{R_k, \, k \in \mathbb{Z}\}$ is sampled from the discrete distribution  $\pi$.


\begin{figure}
\centering
\resizebox*{12cm}{!}{\includegraphics{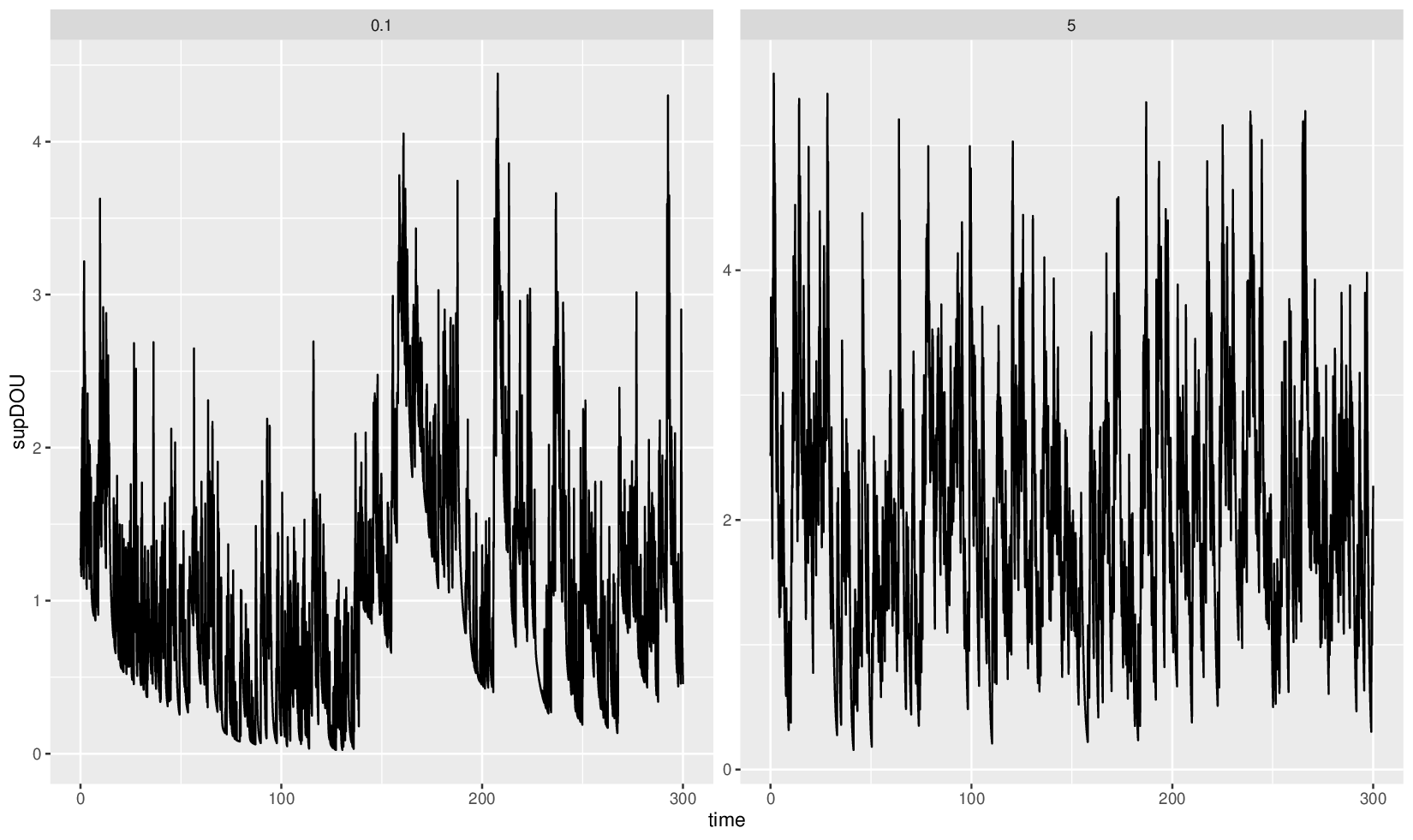}}
\caption{Sample trajectories of the supDOU processes, where $\pi$ is given by the  density of the gamma distribution with $\alpha=0.1$ (left) and $\alpha=5$ (right).}
\label{gamma_supDOU}
\end{figure}

Consider now the case of the continuous measure $\pi$, for example, let  
\begin{equation*}
\pi(dx)=\frac {\alpha^{\alpha+1}} {\Gamma(\alpha+1)}x^\alpha e^{-\alpha x}\mathbbm{1}_{(0,\infty)}(x)dx, \; \alpha>0,
\end{equation*}
as in Example \ref{ex.3}. Then we obtain Algorithm 2 which uses \eqref{supOU:sum_rep} to build sample trajectories of the supDOU process. The infinite sum in \eqref{supOU:sum_rep} is truncated using an additional parameter $T_{min}$, e.g. we can choose $T_{min}=-1000$ for simulation. Note that in this case the supDOU has a long memory whenever $0<\alpha<1$ and a short memory otherwise. We choose $\alpha=0.1$, $\alpha=5$ and $a=1, \theta=2$. Note that both processes on Figure \ref{gamma_supDOU} are stationary processes with the same marginal distribution $GD(2,1)$ but different correlation structures.

\begin{algorithm}
\caption{Simulation of the supDOU process}
\begin{algorithmic}[1]
\Function{supDOU\_Random}{$\alpha$, $\theta$, $T_{min}$, $T$, $dt$}
    \State $n \gets 1 + \lceil T / dt \rceil$
    \State $t \gets \text{vector}(0, dt, 2dt, \ldots, (n-1)dt)$
    \State $X_{sup} \gets \text{rep}(0,n)$
    \State $\tau \gets  T_{min} + \text{ExponentialRandom}(\theta)$
     \State $N\gets  1$
    \State $arrivals[1] \gets  \tau$
    \While{$\tau\leq T$}
        \State $\tau \gets \tau + \text{ExponentialRandom}(\theta)$
         \State $N \gets  N+1$
        \State  $arrivals[N] \gets  \tau$
    \EndWhile

    \State $R \gets$ GammaRandom($N$, shape $= \alpha + 1$, rate $= \alpha$)
    
    \For{$i \gets 1$ to $n$}
        \State $maxIndex \gets \max\{j : arrivals[j] \leq t[i]\}$
        \State $X_{sup}[i] \gets \sum_{k =1}^{ maxIndex} \exp(-R[k] * (t[i] - arrivals[k]))$
    \EndFor
    
    \State \Return $X_{sup}$
\EndFunction
\end{algorithmic}
\end{algorithm}

\section*{Acknowledgments}
Danijel Grahovac was supported by the Croatian Science Foundation under the project Scaling in Stochastic Models (HRZZ-IP-2022-10-8081).
 Nikolai Leonenko (NL) would like to thank for support and hospitality during the programmes “Fractional Differential Equations” (FDE2), “Uncertainly Quantification and Modelling of Materials” (USM), both supported by EPSRC grant EP/R014604/1, and the programme “Stochastic systems for anomalous diffusion” (SSD), supported by EPSRC grant EP/Z000580/1, at Isaac Newton Institute for Mathematical Sciences, Cambridge. NL was partially supported under the ARC Discovery Grant DP220101680 (Australia), Croatian Scientific Foundation grant “Scaling in Stochastic Models” HRZZ-IP-2022-10-8081, grant FAPESP 22/09201-8 (Brazil) and the Taith Research Mobility grant (Wales, Cardiff University).
Also, NL would like to thank Prof. S.A. Molchanov for helpful discussions on the topic in Cambridge (April 2022).



\end{document}